\documentclass{amsart}[15pt]
\usepackage{amsmath,amssymb,cases,amsthm}
\usepackage[dvipdfmx]{graphicx}
\usepackage{float}

\newtheorem{theorem}{Theorem}[section]
\newtheorem{proposition}[theorem]{Proposition}

\newtheorem{lemma}[theorem]{Lemma}

\theoremstyle{definition}
\newtheorem{definition}[theorem]{Definition}
\newtheorem{example}[theorem]{Example}

\theoremstyle{remark}
\newtheorem{remark}[theorem]{Remark}

\numberwithin{equation}{section}

\newcommand{\nky}{\mathbb{A}}
\newcommand{\Fa}{{\rm (F1)}}
\newcommand{\Fb}{{\rm (F2)}}

\usepackage{enumitem}
\newcommand{\Leqq}{\mathrel{\mathpalette\gl@align<}}
\newcommand{\Geqq}{\mathrel{\mathpalette\gl@align>}}
\makeatletter

\begin{document}
\title[A class of nowhere differentiable functions]
{A class of nowhere differentiable functions
satisfying some concavity-type estimate}

\author[Y.~Fujita]{Yasuhiro Fujita}
\address{Department of Mathematics, University of Toyama,
3190 Gofuku, Toyama-shi, Toyama 930-8555, Japan}
\email{yfujita@sci.u-toyama.ac.jp}
\thanks{The first author is supported in part by JSPS KAKENHI  
Nos.~15K04949 and 18K03360}

\author[N.~Hamamuki]{Nao Hamamuki}
\address{Department of Mathematics, Hokkaido University,
Kita 10, Nishi 8, Kita-Ku, Sapporo, Hokkaido, 060-0810, Japan}
\email{hnao@math.sci.hokudai.ac.jp}
\thanks{The second author is supported in part by JSPS KAKENHI
No.~16K17621}

\author[A.~Siconolfi]{Antonio Siconolfi}
\address{Department of Mathematics, Sapienza Universit\`a di Roma, Piazzale Aldo Moro 5, 00185 Roma, Italy}
\email{siconolf@mat.uniroma1.it}
\thanks{}

\author[N.~Yamaguchi]{Norikazu Yamaguchi}
\address{Faculty of Human Development, University of Toyama,
3190 Gofuku, Toyama-shi, Toyama 930-8555, Japan}
\email{norikazu@edu.u-toyama.ac.jp}
\thanks{}



\date{\today}

\begin{abstract}
In this paper, we introduce and investigate a class $\mathcal{P}$
of continuous and periodic functions on $\mathbb{R}$.
The class $\mathcal{P}$ is defined so that
second-order central differences of a function
satisfy some concavity-type estimate.
Although this definition seems to be independent of nowhere differentiable character,
it turns out that each function in $\mathcal{P}$ is nowhere differentiable.
The class $\mathcal{P}$ naturally appear
from both a geometrical viewpoint and an analytic viewpoint.
In fact, we prove that a function belongs to $\mathcal{P}$
if and only if some geometrical inequality holds for a family of parabolas with vertexes on this function.
As its application, we study the behavior of the Hamilton--Jacobi flow
starting from a function in $\mathcal{P}$.
A connection between $\mathcal{P}$
and some functional series is also investigated.
In terms of second-order central differences,
we give a necessary and sufficient condition
so that a function given by the series belongs to $\mathcal{P}$.
This enables us to construct a large number of examples
of functions in $\mathcal{P}$ through an explicit formula.
\end{abstract}

\maketitle

\section{Introduction}
Let us denote by $C_p(\mathbb{R})$
the set of all continuous and periodic functions
$f:\mathbb{R} \to \mathbb{R}$ with period $1$ and $f(0)=0$.
Throughout this paper, we assume that $r$ is an integer such that $r\geq 2$.
Let $\mathbb{N}_0:=\mathbb{N} \cup \{ 0 \}$.

  Our aim of this paper is to introduce and investigate the class $\mathcal{P}$  of functions in $C_p(\mathbb{R})$ defined as follows: Given  a function $f \in C_p(\mathbb{R})$, we consider, for each
$(n,k,y) \in \mathbb{N}_0\times \mathbb {Z}\times (0,1)$,
the first-order forward and backward differences of $f$
at $\frac{k+y}{r^n}$ defined,
respectively,  by
\begin{equation}
 \delta^+_{n,k}(y;f)
   =\frac{f\bigl(\frac{k+1}{r^n}\bigr)-f\bigl(\frac{k+y}{r^n}\bigr)}{\frac{1-y}{r^n}}, \quad
  \delta^-_{n,k}(y;f)
   =\frac{f\bigl(\frac{k+y}{r^n}\bigr)-f\bigl(\frac{k}{r^n}\bigr)}{\frac{y}{r^n}}.
   \label{eqn:22}
\end{equation}

\begin{definition}
\label{defn:21}
Let $c>0$ be a given constant.
A function $f \in C_p({\mathbb R})$ belongs to ${\mathcal P}_c$ if
\begin{equation}
 \delta^+_{n,k}(y;f)-\delta^-_{n,k}(y;f) \leq -c
  \label{eqn:23e}
\end{equation}
for all $(n,k,y) \in \mathbb {N}_0\times \mathbb {Z}\times(0,1)$.
We use the notation $\mathcal{P}=\bigcup_{c>0} \mathcal{P}_c$.
Note that both ${\mathcal P}_c$ and ${\mathcal P}$ depend on
the choice of $r$ though we omit it in our notation.
\end{definition}

Inequality \eqref{eqn:23e} can be  written equivalently as
\begin{equation}
\Delta_{n,k}(y;f)\leq -2c r^n,
\label{eqn:23d}
\end{equation}where $\Delta_{n,k}(y;f)$ is the second-order central difference defined by
\begin{equation}
\Delta_{n,k}(y;f)=2r^n(\delta^+_{n,k}(y;f)-\delta^-_{n,k}(y;f)).
\label{eqn:21b}
\end{equation}
It is well-known that if a function $f:\mathbb{R} \to \mathbb{R}$
is concave and has the second derivative in some interval $I$,
then $f'' \leq 0$ in $I$.
Even if $f$ is not twice differentiable,
a discrete version of the estimate $\Delta_{n,k}(y,f) \leq 0$ still holds.
Thus, the condition \eqref{eqn:23d} can be regarded
as a concavity-type estimate for $f$.
Our definition of $\mathcal{P}$ requires a function
to have the second-order differences which tend to $-\infty$
in the prescribed rate
as $n \to \infty$.

Although Definition \ref{defn:21} seems to be independent
of nowhere differentiable character,
it turns out that each function in $\mathcal{P}$ is nowhere differentiable.
This shows that our concavity-type estimate \eqref{eqn:23d}
is significantly different from a usual concavity
since any concave function is twice differentiable almost everywhere.

We have two reasons to introduce and investigate the class $\mathcal{P}$.
The first reason comes from a geometrical viewpoint.
We show that each function in $\mathcal{P}$
has a geometrical characterization stated as follows:
For any given function
$f\in C_p({\mathbb R})$, let $\{q_f(t,x;z)\}_{z\in{\mathbb R}}$ be the family of parabolas  defined by
\begin{equation}
 q_f(t,x;z)=f(z)+\frac{1}{2t}(x-z)^2,\quad (t,x,z)\in (0,\infty)\times
  \mathbb{R} \times \mathbb{R}. \label{eqn:13}
\end{equation}
Then, we show that a function $f$ in $C_p({\mathbb R})$ belongs to $\mathcal{P}_c$ if and only if $f$ satisfies 
\begin{itemize}
\item[{\Fa}$_c$]
For all $(n,k,y) \in \mathbb{N}_0 \times \mathbb{Z} \times (0,1)$
and $t \geq \frac{1}{2cr^n}$,
\begin{equation}
q_f \left(t,x; \frac{k+y}{r^n} \right)
\geq \min \left\{ q_f \left(t,x; \frac{k}{r^n} \right), \,
q_f \left(t,x; \frac{k+1}{r^n} \right) \right\},\quad x \in \mathbb{R}.
\label{eqn:14}
\end{equation}
\end{itemize}
Inequality \eqref{eqn:14} is a geometrical one
related to position of the three parabolas.
Another interpretation of \eqref{eqn:14} is that
the function $q_f(t,x;\cdot)$ takes a minimum
over the interval $[\frac{k}{r^n},\frac{k+1}{r^n}]$ at the endpoints.

The second reason comes from an analytic viewpoint.
We consider 
the operator
$U: C_p({\mathbb R}) \ni \psi \mapsto U_{\psi} \in C_p({\mathbb R})$
defined by the series
\begin{equation}
 U_\psi(x)= \sum_{j =0}^\infty \frac 1{r^j} \, \psi(r^j x),
  \qquad x \in \mathbb{R}. \label{eqn:31}
\end{equation}
Such a series is known to generate
nowhere differentiable functions
under a suitable condition on $\psi$.
We prove that the condition $U_{\psi} \in \mathcal{P}$ can be
equivalently rephrased by the condition
including the second-order differences of $\psi$.
In fact, we establish
\begin{equation}
\Delta_{n,k}(y;U_{\psi})
=\sum_{j=0}^{n-1}{r^j}\Delta_{n-j,k}(y;\psi)
-\frac{2r^{n}}{y(1-y)}U_{\psi}(y),
\label{eqn:33}
\end{equation}
whenever $\psi \in C_p(\mathbb{R})$ and
$(n,k,y) \in \mathbb {N}_0\times \mathbb {Z}\times(0,1)$.
When $n=0$, the first term of the right-hand side of \eqref{eqn:33}
is interpreted as $0$.
Thus, for a given $c>0$, we see that
$U_{\psi} \in \mathcal{P}_c$ if and only if
the right-hand side of \eqref{eqn:33} is less than or equal to $-2cr^n$ for all $(n,k,y) \in \mathbb {N}_0\times \mathbb {Z}\times(0,1)$.
In other words, the class $\mathcal{P}$
is characterized via the operator $U$.
Besides, making use of \eqref{eqn:33},
we give some sufficient conditions on $\psi$
in order that $U_{\psi} \in \mathcal{P}$.
We show that $U_\psi$ belongs to $\mathcal{P}$
if $\psi$ is concave on $[0,1]$.
Also, even if $\psi$ is not concave on $[0,1]$,
there is the case where $U_{\psi}$ belongs to $\mathcal{P}$
provided that $\psi$ is semiconcave on $[0,1]$
and satisfies some additional assumption.
These simple sufficient conditions enable us
to systematically construct a large number of examples of functions
in the class $\mathcal{P}$ through the explicit formula \eqref{eqn:31}.

A typical example of functions constructed by this procedure
is the generalized Takagi function $\tau_r \in C_p(\mathbb{R})$ defined by
\begin{equation}
\tau_r(x)=U_{d}(x)
=\sum_{j=0}^{\infty}\frac{1}{r^j}d(r^j x),\quad x\in {\mathbb R},
\label{eqn:32}
\end{equation}
where $d\in  C_p({\mathbb R})$ is the distance function to the set $\mathbb{Z}$, that is,
\begin{equation}
d(x)= \min\{|x -z| \,|\, z \in \mathbb{Z} \},
\quad x \in \mathbb{R}.
\label{eqn:varphi}
\end{equation}
The celebrated Takagi function is given by
$\tau_2$. The function $\tau_2$ is equivalent to
the one first constructed by T.~Takagi in 1903,
who showed that $\tau_2$ is nowhere differentiable (see \cite{T}).
Its relevance in analysis, probability theory and number
theory has been widely illustrated by many contributions,
see for instance \cite{T,To,AK,L}.
Since $d$ is concave on $[0,1]$,
we can show that $\tau_r$ belongs to $\mathcal{P}$
for any integer $r \geq 2$.

In connection to {\Fa}$_c$, we also study the behavior
of the Hamilton--Jacobi flow $\{H_t f \}_{t>0}$ starting from $f\in {\mathcal P}$,
where
\begin{equation}
 H_t f(x)=\inf_{z \in \mathbb{R}}q_f(t,x;z), \quad
  (t,x) \in (0, \infty) \times \mathbb{R}.
  \label{eqn:112}
\end{equation}This formula is widely used in the theory of viscosity solutions, and
$H_t f$ is also referred to as an \textit{inf-convolution} of $f$.

There are several papers related to our work.
In \cite{HY}, Hata and Yamaguti proposed a different generalization
of the Tagaki function, the so-called Tagaki class,
which includes not only nowhere differentiable functions,
but also differentiable and even smooth ones.
To analyze this class, they used some functional equations
containing second-order central differences.
Although we also use the second-order central difference
$\Delta_{n,k}(y;f)$ of a function $f\in C_p(\mathbb{R})$,
the frame and the purpose of the investigation of \cite{HY}
are however rather different to ours.
In \cite{Bo,HP, M},
an inequality for approximate midconvexity of the Takagi function
was investigated.
A precise behavior of the flow $\{ H_t \tau \}_{t>0}$
starting from the Takagi function is studied in \cite{FHY}.

The function $U_\psi$ of \eqref{eqn:31} has been considered by many authors.
Cater \cite{Ca} showed that
if $\psi\in C_{p}({\mathbb R})$ is concave on the interval $[0,1]$
and $\psi$ takes its positive maximum over $[0,1]$ at $x=\frac{1}{2}$,
then $U_\psi$ is nowhere differentiable.
Although the connection between the concavity of $\psi$ and $U_\psi$
was already explored in \cite{Ca},
in this paper we show in addition that
the formula \eqref{eqn:31} provides examples of functions in the class $\mathcal{P}$.
Furthermore, we show that $U_\psi$ can belong to $\mathcal{P}$
even if $\psi\in C_{p}({\mathbb R})$ is not concave on $[0,1]$.
Heurteaux \cite{He} gave another sufficient conditions on $\psi\in C_{p}({\mathbb R})$
such that $U_\psi$ is nowhere differentiable.
The set of maximum points in $[0,1]$ of the function $U_{\psi}$
was studied in \cite{FS} for $r=2$.
However, all of the above papers neither
characterize a class of nowhere differentiable functions
nor introduce a class like $\mathcal{P}$.

\medskip
The structure of the present paper is as follows.
In Section \ref{sec:classP} we prove nowhere differentiability
and the geometrical characterization of a function in $\mathcal{P}$.
Section \ref{sec:pathological} is devoted to the formula \eqref{eqn:33}.
We derive some sufficient conditions on $\psi \in \mathcal{P}$ in order that $U_{\psi} \in \mathcal{P}$.
In Section \ref{sec:inf-convolution},
we study how the Hamilton-Jacobi flow
$\{ H_t f \}_{t>0}$ starting from $f\in {\mathcal P}$ behaves.
Section \ref{sec:concluding-remark} contains concluding remarks.

\section{The class $\mathcal{P}$}
\label{sec:classP}

In this section, we state and prove several results
on the class $\mathcal{P}$.
The first result of this section is Theorem \ref{thm:31},
where we prove that each function in $\mathcal{P}$ is nowhere differentiable.
The second result of this section is Theorem \ref{thm:24},
which shows that a function $f$ in $C_p(\mathbb{R})$
belongs to $\mathcal{P}_c$ if and only if $f$ satisfies {\Fa}$_c$.


Since we study periodic functions with period $1$,
we often choose three points
$\frac{k}{r^n}$, $\frac{k+y}{r^n}$, $\frac{k+1}{r^n}$ lying in $[0,1]$.
For this reason, we prepare the set $\nky$
of admissible triplets $(n,k,y)$ as
\[ \nky:=\{ (n,k,y) \mid n \in \mathbb{N}_0, \,
k \in \{0,1,2,3,\dots , r^n-1\}, \, y \in (0,1) \}. \]
For any $(n,k,y) \in \nky$ we have
$[\frac{k}{r^n},\frac{k+1}{r^n}] \subset [0,1]$.
For a constant $c>0$, note that
$f \in C_p({\mathbb R})$ belongs to $\mathcal{P}_c$ if and only if
\eqref{eqn:23e} is satisfied for all $(n,k,y) \in \nky$.

We first derive a fundamental inequality for $f \in \mathcal{P}$.
For $f\in C_p(\mathbb{R})$, we see by \eqref{eqn:21b} that
\begin{equation}
 \Delta_{0,0}(y;f)=\frac{-2f(y)}{y(1-y)},\quad y\in (0,1).
  \label{eqn:D00}
\end{equation}
Thus, for $c>0$ and $y \in (0,1)$,
we have $\Delta_{0,0}(y;f) \leq -2c$ if and only if
\begin{equation}
 cy(1-y) \leq f(y).
  \label{eqn:cyf}
\end{equation}
Therefore we see that
every $f \in \mathcal{P}_c$ satisfies \eqref{eqn:cyf}
for any $y\in (0,1)$.
In particular, when $f\in \mathcal{P}$, we have $f>0$ in $(0,1)$.

Now, we show that each function in $\mathcal{P}$
is nowhere differentiable.
In what follows we write $[z]$ for $z\in \mathbb{R}$
to indicate the largest integer not exceeding $z$.
We denote by $\mathbb{Q}_r$ the set of all rational numbers
that can be written as $\frac{k}{r^n}$ for some
$n \in \mathbb{N}$ and $k\in {\mathbb Z}$.

\begin{theorem}\label{thm:31}
Each function in ${\mathcal P}$ is nowhere differentiable in $\mathbb{R}$.
\end{theorem}

\begin{proof}
Fix $c>0$.
Suppose that $f \in \mathcal{P}_c$ is differentiable at some point $x \in [0,1]$.

We set $k_n=[r^n x]$ for each $n \in \mathbb{N}$.
Also, set $y_n=y$ if $x\in \mathbb{Q}_r$
and $y_n=r^n x -[r^n x]$ if $x \not \in \mathbb{Q}_r$,
where $y \in (0,1)$ is an arbitrary constant.
We claim that $\delta^{\pm}_{n,k_n}(y_n;f) \to f'(x)$ as $n \to \infty$.
This gives a contradiction
since taking the limit $n \to \infty$
in \eqref{eqn:23e} along these $k_n$ and $y_n$ implies that $0 \leq -c$.

When $x\in \mathbb{Q}_r$,
we have $[r^n x]=r^n x$ for $n \in \mathbb{N}$ large.
In fact, since $x\in \mathbb{Q}_r$,
there are $n_0 \in \mathbb{N}_0$ and $k_0 \in \mathbb{Z}$
such that $x=\frac{k_0}{r^{n_0}}$, so that
$r^n x=k_0r^{n-n_0}\in {\mathbb N}$ if $n \geq n_0$.
For $n\geq n_0$ we find that
\begin{align*}
 \delta^+_{n,k_n}(y_n;f)
 &=\frac{f \bigl( x+\frac{1}{r^n} \bigr)
 -f \bigl(x+\frac{y}{r^n} \bigr)}{\frac{1-y}{r^n}}
 =\frac{f\bigl(x+\frac{1}{r^n}\bigr)-f(x)}{\frac{1}{r^n}\,(1-y)}
 -\frac{f\bigl(x+\frac{y}{r^n}\bigr)-f(x)}{\frac{y}{r^n}\,\frac{1-y}{y}}\\
 &\to \frac{f'(x)}{1-y}-y\frac{f'(x)}{1-y}
 =f'(x) \quad (n \to \infty).
\end{align*}
In the same manner,
we deduce that $\delta^-_{n,k_n}(y_n;f) \to f'(x)$ as $n \to \infty$.

Next, let $x\not\in {\mathbb Q}_r$.
We then have $[r^n x]< r^n x< [r^n x]+1$ for each $n \in \mathbb{N}$.
This implies that $y_n \in (0,1)$ for each $n \in \mathbb{N}$ and that
$\frac{[r^nx]}{r^n} \to x$ as $n \to \infty$.
Thus,
\[
 \delta^+_{n,k_n}(y_n;f)
=\frac{f \bigl( \frac{[r^n x]+1}{r^n} \bigr)-f(x)}{\frac{[r^n x]+1}{r^n}-x}
\to f'(x) \quad (n \to \infty).
\]
Similarly, it follows that $\delta^-_{n,k_n}(y_n;f) \to f'(x)$.
This completes the proof.
\end{proof}

Next, we show that a function $f$ in $C_p({\mathbb R})$ belongs to $\mathcal{P}_c$
if and only if $f$ satisfies {\Fa}$_c$.
To prove this, the following proposition is essential:

\begin{proposition}
 \label{prop:23}
 Let $(n,k,y) \in \nky$ and $t\in (0,\infty)$.
 Then, for any $f\in C_{p}({\mathbb R})$,
 inequality \eqref{eqn:14} holds if and only if
 \begin{equation}
  \Delta_{n,k}(y;f)\leq - \frac{1}{t}.
   \label{eqn:24}
 \end{equation}
\end{proposition}

\begin{proof}
 Fix $(n,k,y) \in \nky$ and $t\in (0,\infty)$.
 Let $x_1(n,k,y,t)$ be the  unique solution of the equation
 \[
 q_f \left(t,x; \frac{k+y}{r^n} \right)
 =q_f\left(t,x; \frac{k}{r^n} \right).
 \]
By direct calculation,
\begin{equation}
 x_1(n,k,y,t)=\frac{k}{r^n}+\frac{y}{2r^n}+ t \delta^-_{n,k}(y;f).
  \label{eqn:26}
\end{equation}
Then, we have
\[
 \begin{cases}
  q_f \left( t,x;\dfrac{k}{r^n} \right)
  \leq q_f \left(t,x;\dfrac{k+y}{r^n} \right),
  & x \leq x_1(n,k,y,t), \medskip \\
  q_f \left(t,x;\dfrac{k}{r^n} \right)
  > q_f \left(t,x;\dfrac{k+y}{r^n} \right),
  & x_1(n,k,y,t) < x.
 \end{cases}
\]
Similarly, the unique solution $x_2(n,k,y,t)$ of the equation
\[
 q_f \left(t,x; \frac{k+y}{r^n} \right)
 =q_f \left(t,x; \frac{k+1}{r^n} \right)
\]
is given by
\begin{equation}
x_2(n,k,y,t)=\frac{k}{r^{n}}+\frac{1+y}{2r^{n}}+ t \delta^+_{n,k}(y;f).
\label{eqn:29}
\end{equation}
Furthermore,
\[
 \begin{cases}
  q_f \left(t,x;\dfrac{k+y}{r^n} \right)
  \geq q_f \left(t,x;\dfrac{k+1}{r^n} \right),
  & x_2(n,k,y,t) \leq x, \medskip \\
  q_f \left(t,x;\dfrac{k+y}{2r^n} \right)
  < q_f \left(t,x; \dfrac{k+1}{r^n} \right),
  & x< x_2(n,k,y,t).
 \end{cases}
\]
Then, a geometrical investigation shows that
inequality \eqref{eqn:14} holds if and only if
\begin{equation}
 x_1(n,k,y,t)\geq x_2(n,k,y,t).\label{eqn:211}
\end{equation}
 By \eqref{eqn:26} and \eqref{eqn:29},
 we see that inequality \eqref{eqn:211} holds if and only if
\[
 \delta^-_{n,k}(y;f)-\delta^+_{n,k}(y;f) \geq \frac{1}{2r^n t}.
\]
The desired inequality follows immediately from \eqref{eqn:21b}.
\end{proof}

Now, we state the second result of this section.

\begin{theorem}
 \label{thm:24}
 Let $f \in C_{p}(\mathbb{R})$ and let $c>0$ be a constant.
 Then, $f$ satisfies {\Fa}$_c$
 if and only if  $f \in \mathcal{P}_c$.
\end{theorem}

\begin{proof}
Assume first that $f \in \mathcal{P}_c$.
Fix $(n,k,y) \in \nky$ and $t \geq \frac{1}{2cr^n}$ arbitrarily.
By \eqref{eqn:23d} and \eqref{eqn:21b}, we have
\[
 \Delta_{n,k}(y;f) \leq -2cr^{n}\leq -\frac{1}{t},
\]
and so \eqref{eqn:14} holds by Proposition~\ref{prop:23}.
Thus we see that $f$ satisfies {\Fa}$_c$.

Next, assume that {\Fa}$_c$ holds.
Then, by Proposition~\ref{prop:23}, we see that
\[
 \Delta_{n,k}(y;f)\leq -\frac{1}{t}
\]
for all $(n,k,y) \in \nky$ and $t \geq \frac{1}{2cr^n}$.
Letting $t=\frac{1}{2cr^n}$,
we conclude that $f \in \mathcal{P}_c$.
\end{proof}

\section{Functions $U_{\psi}$ and $\mathcal{P}$}
\label{sec:pathological}

In this section, we give sufficient conditions
on $\psi \in C_p(\mathbb{R})$ in order that $U_{\psi} \in \mathcal{P}$,
where $U$ is the operator defined by \eqref{eqn:31}.
The results enable us to generate a large number of functions
in $\mathcal{P}$ through the explicit formula \eqref{eqn:31}.
We also give some examples of $\psi \in C_p(\mathbb{R})$
for which 
$U_{\psi} \not \in \mathcal{P}$.

The following theorem provides a representation
of $\Delta_{n,k}(U_{\psi};y)$ in terms of $\Delta_{n,k}(\psi;y)$,
which plays a crucial role to study if $U_{\psi} \in \mathcal{P}$.
Note that, for every $\psi \in C_p(\mathbb{R})$,
we have $U_{\psi} \in C_p(\mathbb{R})$
and $U_\psi(0)=0$ by the definition of $U_{\psi}$.

\begin{theorem}\label{thm:34}
Let $\psi \in C_p(\mathbb{R})$.
Then, \eqref{eqn:33} holds for each $(n,k,y) \in \nky$.
When $n=0$,
the first term of the right-hand side of \eqref{eqn:33} is interpreted as $0$.
\end{theorem}

\begin{proof}
Let $(n,k,y) \in \nky$.
When $n=0$, we have $k=0$,
so that \eqref{eqn:33} follows from \eqref{eqn:D00}
since $U_\psi(0)=0$.
If $n \geq 1$, then
\begin{align*}
 &U_{\psi}\left( \frac{k+y}{r^n} \right)
 -\sum_{j=0}^{n-1}\frac{1}{r^j}\psi \left( \frac{k+y}{r^{n-j}} \right) \\
 &=\sum_{j=n}^{\infty} \frac{1}{r^j} \psi \left( r^{j-n}(k+y) \right)
 =\sum_{j=n}^{\infty} \frac{1}{r^j} \psi \left( r^{j-n}y \right)
 =\frac{1}{r^n}U_{\psi}(y).
\end{align*}
This is valid even for $y=0$ and $y=1$.
Since $U_{\psi}(0)= U_{\psi}(1)=0$, we have
\[
 U_{\psi} \left( \frac{k}{r^n} \right)
 =\sum_{j=0}^{n-1} \frac{1}{r^j} \psi \left(\frac{k}{r^{n-j}} \right), \quad
 U_{\psi} \left( \frac{k+1}{r^n} \right)
 =\sum_{j=0}^{n-1} \frac{1}{r^j} \psi \left(\frac{k+1}{r^{n-j}} \right).
\]
We therefore have
\begin{align*}
\Delta_{n,k}(y;U_{\psi})
&= 2r^{n}\left[
\frac{U_{\psi}\bigl(\frac{k+1}{r^n}\bigr)-U_{\psi}\bigl(\frac{k+y}{r^n}\bigr)}{\frac{1-y}{r^n}}-\frac{U_{\psi}\bigl(\frac{k+y}{r^n}\bigr)-U_{\psi}\bigl(\frac{k}{r^n}\bigr)}{\frac{y}{r^n}}\right]\\
&= 2r^{n}\left[
\frac{\sum_{j=0}^{n-1}\frac{1}{r^j}\Bigl(\psi\bigl(\frac{k+1}{r^{n-j}}\bigr)-\psi\bigl(\frac{k+y}{r^{n-j}}\bigr)\Bigr) -\frac{1}{r^n}U_{\psi}(y)}{\frac{1-y}{r^n}}\right.\\
&\qquad \qquad - \left. \frac{\sum_{j=0}^{n-1}\frac{1}{r^j}\Bigl(\psi\bigl(\frac{k+y}{r^{n-j}}\bigr)-\psi\bigl(\frac{k}{r^{n-j}}\bigr)\Bigr) +\frac{1}{r^n}U_{\psi}(y)}{\frac{y}{r^n}}\right]\\
&= \sum_{j=0}^{n-1}r^j\ 2r^{n-j}\left[
\frac{\psi\bigl(\frac{k+1}{r^{n-j}}\bigr)-\psi\bigl(\frac{k+y}{r^{n-j}}\bigr)}{\frac{1-y}{r^{n-j}}}-\frac{\psi\bigl(\frac{k+y}{r^{n-j}}\bigr)-\psi\bigl(\frac{k}{r^{n-j}}\bigr)}{\frac{y}{r^{n-j}}}\right]\\
&\qquad\qquad
-\frac{2r^{n}}{y(1-y)}U_{\psi}(y)\\
&=\sum_{j=0}^{n-1}{r^j}\Delta_{n-j,k}(y;\psi)-\frac{2r^{n}}{y(1-y)}U_{\psi}(y).
\end{align*}
This implies \eqref{eqn:33}.
\end{proof}


Applying Theorem \ref{thm:34},
we derive some sufficient conditions
on $\psi \in C_p(\mathbb{R})$ that guarantee $U_{\psi} \in \mathcal{P}$.
As a typical result, it turns out that
$U_{\psi} \in \mathcal{P}$
if $\psi$ is concave in $[0,1]$ and positive in $(0,1)$.

Let us recall a notion of concavity.
A function $g:[0,1] \to \mathbb{R}$ is said to be concave on $[0,1]$
if the inequality
\[ \lambda g(x)+(1-\lambda)g(y)\leq g(\lambda x+(1-\lambda)y) \]
holds for all $x,y\in [0,1]$ and $\lambda\in [0,1]$.
If the reversed inequality holds, then $g$ is said to be convex.
For a constant $\alpha \geq 0$,
a function $g$ on $[0,1]$ is said to be $\alpha$-semiconcave on $[0,1]$
if $g(x)+ \frac{\alpha}{2}x(1-x)$ is concave on $[0,1]$.
This is equivalent to the condition that
$g(x)- \frac{\alpha}{2}x^2$ is concave on $[0,1]$.

\begin{remark}\label{rem:32}
\begin{enumerate}[label=(\roman*),leftmargin=*]
 \item
	  Let $\psi \in C_p(\mathbb{R})$
	  and assume that $\psi$ is concave on some interval $I$.
	  Then it is easy to see that
	  $\Delta_{n,k}(y;\psi) \leq 0$
	  for all $(n,k,y) \in \mathbb{N}_0\times \mathbb{Z}\times (0,1)$
	  such that $[\frac{k}{r^n},\frac{k+1}{r^n}] \subset I$.
	  More generally, if $\psi \in C_p({\mathbb R})$ is $\alpha$-semiconcave on $I$,
	  then we have $\Delta_{n,k}(y;\psi) \leq \alpha$
	  for all $(n,k,y) \in \mathbb{N}_0\times \mathbb{Z}\times (0,1)$
	  such that $[\frac{k}{r^n},\frac{k+1}{r^n}] \subset I$.
	  The reversed inequalities hold for ($\alpha$-semi)convex functions.
  \item If $\psi \in C_p({\mathbb R})$ is concave on $[0,1]$,
	   then we have $\Delta_{n,k}(y,\psi) \leq 0$ for all $(n,k,y) \in \nky$ by (i).
	   However, the converse is not true in general:
	   that is, even if $\Delta_{n,k}(y,\psi) \leq 0$ for all $(n,k,y) \in \nky$,
	   we cannot say that $\psi$ is  concave on $[0,1]$.
	   Every $f\in {\mathcal P}$ gives a counterexample to this.
	   In fact, $\Delta_{n,k}(y,f)\leq 0$ for all $(n,k,y) \in \nky$,
	   but $f$ is never concave on $[0,1]$ by Theorem~\ref{thm:31},
	   since a concave function must be differentiable almost everywhere.
\end{enumerate}
\end{remark}

We first prepare inequalities involving $U_{\psi}$ and
the generalized Takagi function $\tau_r$ defined in \eqref{eqn:32}.
Recall that $d$ is the distance function given by \eqref{eqn:varphi}.

\begin{lemma}\label{lem:33}
Let $\psi \in C_p(\mathbb{R})$.
Assume that there exists a constant $m>0$
such that $m d(x) \leq \psi(x)$ for all $x \in [0,1]$.
Then, we have
\begin{equation}
\frac{m r}{r-1} x(1-x)
\leq m \tau_r(x)
\leq U_{\psi}(x), \quad x \in [0,1].
\label{eqn:317}
\end{equation}
\end{lemma}

\begin{proof}
It follows from our assumption that $m d(r^jx)\leq \psi(r^jx)$
for all $x\in [0,1]$ and $j \in \mathbb{N}_0$.
Thus, $m \tau_r(x)\leq U_{\psi}(x)$ by taking the sum.

It remains to prove that
\begin{equation}
\frac{r}{r-1} x(1-x) \leq \tau_r(x), \quad x \in [0,1].
\label{eqn:318}
\end{equation}
Let
\[
 F(x)=d(x)+\frac{1}{r}d(rx),\quad
 G(x)=\frac{r}{r-1}x(1-x), \quad x \in [0,1].
\]
Since $F \leq \tau_r$,
it suffices to show that $G(x) \leq F(x)$ for $x\in [0,1]$.
As $F$ and $G$ are symmetric about $x=\frac{1}{2}$,
we may assume that $x \in [0,\frac{1}{2}]$.
Note that
\[
 F(x)=2x \
 \left( 0 \leq x \leq \frac{1}{2r} \right), \
 F(x)=\frac{1}{r} \
 \left( \frac{1}{2r} \leq x \leq \frac{1}{r} \right), \
 F(x) \geq x \
 \left( \frac{1}{r} \leq x \leq \frac{1}{2} \right).
\]
When $0 \leq x \leq \frac{1}{r}$, we have
\[
 G(x) \leq G\left( \frac{1}{r} \right)=\frac{1}{r}, \quad
 G(x) \leq \frac{r}{r-1}x(1-0) \leq 2x.
\]
Thus $G(x) \leq F(x)$.
Next, let $\frac{1}{r} \leq x \leq \frac{1}{2}$.
Then,
\[
 G(x) \leq \frac{r}{r-1}x \left( 1-\frac{1}{r} \right)=x \leq F(x).
\]
Hence, we conclude \eqref{eqn:318}.
\end{proof}

\begin{remark}\label{rem:3200}
Assume that $\psi \in C_p(\mathbb{R})$ is concave in $[0,1]$
and $\psi>0$ in $(0,1)$.
Then, we have
\begin{equation}
2 \psi \left( \frac{1}{2} \right) d(x) \leq \psi(x),
\quad x \in [0,1],
\label{eq:dpsi}
\end{equation}
and thus $\psi$ satisfies the assumption in Lemma \ref{lem:33}
for $m= 2\psi (\frac{1}{2})$.
Indeed, by the concavity of $\psi$, its graph lies above
the segment connecting $(0,\psi(0))$ and $(\frac{1}{2},\psi(\frac{1}{2}))$
and the segment connecting $(\frac{1}{2},\psi(\frac{1}{2}))$ and $(1,\psi(1))$.
This shows \eqref{eq:dpsi} since $\psi(0)=\psi(1)=0$.
\end{remark}

Now, we state the main result of this section.

\begin{theorem}\label{thm:35}
Let $\psi \in C_p(\mathbb{R})$.
Assume that there exist two constants
$m>0$ and $\alpha \geq 0$ such that
\begin{enumerate}[label={\rm (\roman*)}]
 \item$m d(x) \leq \psi(x)$ for all $x\in [0,1]$.
 \item $\Delta_{n,k}(y;\psi) \leq \alpha$ for all $(n,k,y) \in \nky$.
\end{enumerate}
If $2mr > \alpha$, then $U_{\psi} \in \mathcal{P}_c$ with $c=\frac{2mr-\alpha}{2(r-1)}$.
\end{theorem}

\begin{proof}
Let us derive $\Delta_{n,k}(y;U_{\psi}) \leq -2cr^n$
for a fixed $(n,k,y) \in \nky$.
From Lemma \ref{lem:33} it follows that
\[ -\frac{2r^{n}}{y(1-y)}U_{\psi}(y)
\leq -\frac{2m r^{n+1}}{r-1}. \]
If $n=0$, we see by \eqref{eqn:D00} that
$\Delta_{0,0} (y;U_{\theta}) \leq -\frac{2mr}{r-1}<-2c$.
For $n \geq 1$ we have
\[
\sum_{j=0}^{n-1}r^j \Delta_{n-j,k}(y;\psi)
\leq \sum_{j=0}^{n-1} r^j \alpha
=\alpha \cdot \frac{r^n-1}{r-1}
<\alpha \cdot \frac{r^n}{r-1}.
\]
Thus, by \eqref{eqn:33}
\[
\Delta_{n,k}(y;U_\psi)
\leq \alpha \cdot \frac{r^n}{r-1} - \frac{2m r^{n+1}}{r-1}
=-2cr^n,
\]
which proves the theorem.
\end{proof}

Let us denote by $E$ the set of $\psi \in C_p(\mathbb{R})$
satisfying (i) and (ii) in Theorem \ref{thm:35}
for some $m>0$ and $\alpha \geq 0$ with $2mr > \alpha$.
Theorem \ref{thm:35} asserts that $U_{\psi} \in \mathcal{P}$ for every $\psi \in E$.
We give typical classes that are included in $E$.

\begin{proposition}\label{prop:twosets}
The set $E$ includes the following two sets:
\begin{itemize}
\item[$(1)$]
$\mathit{SC}_0:=\{ \psi \in C_p(\mathbb{R}) \mid \mbox{$\psi$ is concave in $[0,1]$ and $\psi>0$ in $(0,1)$} \}$.
\item[$(2)$]
$\mathcal{P}$.
\end{itemize}
\end{proposition}

\begin{proof}
(1)
Let $\psi \in \mathit{SC}_0$.
It follows from Remark \ref{rem:3200}
that $\psi$ satisfies Theorem \ref{thm:35}-(i)
for $m=2\psi(\frac{1}{2})$,
while we can take $\alpha=0$ in Theorem \ref{thm:35}-(ii) by Remark \ref{rem:32}-(i).
Since $2mr > \alpha$, we have $\psi \in E$
and $U_{\psi} \in \mathcal{P}_c$ with
$c=\frac{2r}{r-1} \psi ( \frac{1}{2} )$.

\medskip

(2)
Let $\psi \in \mathcal{P}_c$ for some $c>0$.
By \eqref{eqn:cyf}, we can take $m=c$ in Theorem \ref{thm:35}-(i).
We also take $\alpha=0$ in Theorem \ref{thm:35}-(ii)
by the definition of $ \mathcal{P}_c$.
Since $2mr > \alpha$, we conclude that $\psi \in E$
and $U_\psi\in \mathcal{P}_{c'}$ with $c'=\frac{cr}{r-1}$.
\end{proof}

Note that the two sets
$\mathit{SC}_0$ and $\mathcal{P}$ above are mutually disjoint,
since a concave function is differentiable almost everywhere.
Also, if $\psi$ belongs to $ \mathcal{P}$,
then $U_\psi$ also belongs to $ \mathcal{P}$
since $\mathcal{P} \subset E$ by Proposition \ref{prop:twosets}-(2).
Thus, $\mathcal{P}$ is an invariant set under the operator $U$.

\begin{remark}\label{rem:36}
By Proposition \ref{prop:twosets}-(1) and its proof,
we see that the generalized Takagi function $\tau_r$
belongs to $\mathcal{P}_c$ with $c=\frac{r}{r-1}$
since $d \in C_p({\mathbb R})$ is concave in $[0,1]$
and $d(\frac{1}{2})=\frac{1}{2}$.
In particular, the Takagi function $\tau_2$
is in $\mathcal{P}_2$ for $r=2$.
\end{remark}


If $\psi \in C_p(\mathbb{R})$ is $\alpha$-semiconcave in $[0,1]$,
then (ii) in Theorem \ref{thm:35} is fulfilled
by Remark \ref{rem:32}-(i).
However, (i) does not hold in general even if $\psi>0$ in $(0,1)$.
One may then wonder
if $U_{\psi}$ belongs to $\mathcal{P}$ for $\psi$ in
\[ \mathit{SC}_{\alpha}
:=\{ \psi \in C_p(\mathbb{R}) \mid
\mbox{$\psi$ is $\alpha$-semiconcave in $[0,1]$ and $\psi>0$ in $(0,1)$} \} \]
with $\alpha>0$.
The answer is no.
Besides, $U_{\psi}$ for $\psi \in \mathit{SC}_{\alpha}$
does not necessarily possess nowhere differentiable character.
Namely, for every $\alpha>0$
there are the following three examples of $\psi \in \mathit{SC}_{\alpha}$:
\begin{itemize}
\item[(A)]
$U_{\psi} \in \mathcal{P}$
and $\psi \not \in \mathit{SC}_0$.
\item[(B)]
$U_{\psi} \not \in \mathcal{P}$
and $U_{\psi}$ is nowhere differentiable in $[0,1]$.
\item[(C)]
$U_{\psi} \not \in \mathcal{P}$
and $U_{\psi} \in C^{\infty}((0,1))$.
\end{itemize}

Let us give an example of $\psi \in \mathit{SC}_{\alpha}$
satisfying each (A)--(C).

\begin{example}\label{ex:308}
For constants $a, b>0$, let $\psi_0 =ad +bd^2\in C_p(\mathbb{R})$. Then, $\psi_0$ is  not concave on $[0,1]$ but $2b$-semiconcave on $[0,1]$. In addition, when $ar>b$, $U_{\psi_0} \in \mathcal{P}$.
We thus obtain a function satisfying (A).

Indeed, since $\psi_0(x)= ax+bx^2$ on $[0,\frac{1}{2}]$,
$\psi_0$ is not concave on $[0,1]$.
Also, we have $\psi_0(x)+bx(1-x)=(a+b)d(x)$ on $[0,1]$,
and so $\psi_0$ is $2b$-semiconcave on $[0,1]$.
Finally, since $\psi_0 \geq ad$ on $[0,1]$,
we can take $m=a$ and $\alpha=2b$ in Theorem \ref{thm:35}.
Thus, $\psi_0 \in E$ and so $U_{\psi_0}\in \mathcal{P}$.

This example also shows that $\mathit{SC}_0 \cup \mathcal{P} \subsetneq E$.
\end{example}


Let us next discuss the example of (B).
Let $\theta \in C_p(\mathbb{R})$ be a function such that
\[ \theta(x)=x^2 \mbox{ for } x \in \left[ 0,\frac{1}{r} \right],
\quad \theta \in C^2(\mathbb{R}),
\quad \theta>0 \mbox{ in } (0,1). \]
We now apply \cite[Theorem 3.1]{He},
which asserts that,
if $\psi \in C_p(\mathbb{R}) \cap C^1(\mathbb{R})$
and $\psi'$ is H\"older continuous in $\mathbb{R}$,
then
$U_{\psi}$ is nowhere differentiable in $\mathbb{R}$.
Since $\theta$ satisfies these conditions,
we deduce that $U_\theta$ is nowhere differentiable in $\mathbb{R}$.
However, $U_\theta$ does not belong to $\mathcal{P}$ as shown below.


\begin{theorem}\label{thm:310}
$\Delta_{n,0}(\frac{1}{r};U_{\theta})=-\frac{2}{r-1}$
for each $n \in \mathbb{N}_0$.
Thus, $U_{\theta} \not \in \mathcal{P}$.
\end{theorem}

\begin{proof}
Let $n \in \mathbb{N}_0$.
We have
\[
U_{\theta} \left( \frac{1}{r} \right)
=\sum_{j=0}^{\infty} \frac{1}{r^j} \theta(r^{j-1})
=\theta(r^{-1})=\frac{1}{r^2}.
\]
Thus,
\[
\left. \frac{2r^n}{y(1-y)} U_{\theta}(y) \right|_{y=\frac{1}{r}}
=\frac{2r^n}{r-1}.
\]
When $n=0$, this and \eqref{eqn:D00} shows that
$\Delta_{0,0} (\frac{1}{r};U_{\theta})=-\frac{2}{r-1}$.
Let $n \geq 1$.
Since $\Delta_{m,0}\bigl(\frac1r,{\theta})=2$ for any $m \in \mathbb{N}$, it follows from Theorem \ref{thm:34} that
\begin{align*}
\Delta_{n,0} \left( \frac{1}{r};U_{\theta} \right)
&= \sum_{j=0}^{n-1}{r^j} \Delta_{n-j,0} \left( \frac{1}{r};\theta \right)
-\left. \frac{2r^n}{y(1-y)} U_{\theta}(y) \right|_{y=\frac{1}{r}} \\
&=2 \sum_{j=0}^{n-1}{r^j}
-\frac{2r^n}{r-1}=-\frac{2}{r-1}.
\end{align*}
The proof is complete.
\end{proof}

Let $\alpha>0$.
Since $\theta \in C^2(\mathbb{R})$,
we have $\varepsilon \theta \in \mathit{SC}_{\alpha}$
if $\varepsilon>0$ is sufficiently small.
Also, it is easy to see that
$U_{\varepsilon \theta}$ is still nowhere differentiable and
$U_{\varepsilon \theta} \not \in \mathcal{P}$.
We thus obtain a function satisfying (B).

\begin{example}
Let us give an example of a function satisfying (C).
Define
\[ \psi(x)=|\sin (\pi x)|-\frac{1}{r} |\sin (\pi r x)| \in C_p(\mathbb{R}). \]
Then, by the definition of $U_{\psi}$, we easily see that
$U_{\psi}(x)=|\sin (\pi x)| \in C_p(\mathbb{R})$.
Thus $U_{\psi} \in C^{\infty}((0,1))$
and in particular $U_{\psi} \not \in \mathcal{P}$ as required in (C).

Let us next check that $\psi \in \mathit{SC}_{\alpha}$ for some $\alpha>0$.
The positivity of $\psi$ in $(0,1)$ follows
from straightforward calculation, and so we omit the proof.
Next, since functions
$\frac{1}{r} \sin (\pi r x)$ and $-\frac{1}{r} \sin (\pi r x)$
are semiconcave,
the minimum of them $-\frac{1}{r} |\sin (\pi r x)|$
is also semiconcave.
Therefore, $\psi$ being the sum of two semiconcave functions in $[0,1]$
is semiconcave in $[0,1]$.

Similarly to the previous example, for a given $\alpha>0$,
we have $\varepsilon \psi \in \mathit{SC}_{\alpha}$
if $\varepsilon>0$ is sufficiently small.
A function satisfying (C) has thus been obtained.
\end{example}

We conclude this section by studying
if a Weierstrass type function belongs $\mathcal{P}$.

\begin{example}
The famous Weierstrass function $W$ is given by
\[
W(x)=\sum_{j=0}^{\infty}a^j \rho(b^j x),\quad \rho(x)=\cos (\pi x),
\]
where $a \in (0,1)$ and $b$ is an odd integer
with $ab>1+\frac{3\pi}{2}$.
Note that $\rho$ is continuous and periodic on $\mathbb{R}$
with period $2$ and $\rho(0) \neq 0$.
Since we consider functions $\psi$ in $C_p(\mathbb{R})$
with $\psi(0)=0$ in this paper,
we study $U_{\eta}$ for
$\eta(x)=\sin (2\pi x) \in C_p(\mathbb{R})$
instead of $W$.
By Hardy \cite{Ha},
it is shown that $U_{\eta}$ is nowhere differentiable.
We also remark that
$\eta$ possesses a balance of convexity and concavity properties,
since it is concave on $[0,\frac12]$ and convex on $[\frac12,1]$.

We claim that $U_{\eta}$ does not belong to $\mathcal{P}$.
In fact, noting that
$\eta( \frac{r^j}{2})=\sin ( \pi r^j)=0$ for all $j \in \mathbb{N}_0$,
we see that
$U_{\eta}(\frac{1}{2})=0$ by the definition of $U_{\eta}$.
This implies that $U_{\eta} \not \in \mathcal{P}$
since, if $U_{\eta} \in \mathcal{P}$,
we have $U_\eta> 0$ in $(0,1)$ by \eqref{eqn:cyf}.
\end{example}

\section{The behavior of $\{ H_tf \}_{t>0}$ for $f\in \mathcal{P}$}
\label{sec:inf-convolution}

In this section
we consider the behavior of the Hamilton-Jacobi flow $\{H_tf \}_{t>0}$
for $f\in \mathcal{P}$,
where $H_tf$ is the function defined by \eqref{eqn:112}.
It is known that $H_t f$ belongs to $C_p({\mathbb R})$
and uniformly approximates $f$ as $t$ goes to $0$
(see \cite[Chapter 3.5]{CS}).
Also, $H_t f$ is a unique viscosity solution
of the initial value problem of the Hamilton--Jacobi equation:
\begin{equation}
 \begin{cases}
  u_t(t,x)+\dfrac{1}{2}\left(u_x(t,x)\right)^2=0,
  &(t,x) \in (0,\infty) \times \mathbb{R}, \\
  u(0,x)=f(x),
  & x \in \mathbb{R}
 \end{cases}
\label{eqn:113}
\end{equation}
(cf. \cite{CIL}).
Here, $u_t(t,x)=\frac{\partial u}{\partial t}(t,x)$
and $u_x(t,x)=\frac{\partial u}{\partial x}(t,x)$.

First of all,
we prove that the range of $z$ in \eqref{eqn:112} can be reduced.

\begin{lemma}
 \label{lem:y01}
 Let $f \in C_p(\mathbb{R})$. 
 If $f(z)\geq 0$ for all $z\in [0,1]$, then
\begin{equation}
 H_t f(x)=\min_{z \in [0,1]}q_f(t,x;z),
  \quad (t,x) \in (0,\infty) \times [0,1].
  \label{eqn:11r}
\end{equation}
\end{lemma}

\begin{proof}
Fix $(t,x) \in (0,\infty) \times [0,1]$.
We first let $z<0$.
Since $f(z) \geq 0$,
the geometrical investigation implies that $q_f(t,x;z)> q_f(t,x;0)$.
Thus, the minimum in \eqref{eqn:112} is never attained for $z<0$.
The same arguments show that $z>1$ is not a minimizer of \eqref{eqn:112},
and hence \eqref{eqn:11r} holds.
\end{proof}




Now, we state the main result of this section.

\begin{theorem}\label{thm:25}
Let $f\in {\mathcal P}_c$ for $c>0$.
Then, the following holds:
\begin{itemize}
\item[{\Fb}$_c$]
			 For all $n\in \mathbb{N}_0$,
			 \begin{equation}
			  H_t f(x)
			   =\min_{k \in \{0,1,2,3,\dots,r^n\}}
			   q_f \left(t,x; \frac{k}{r^n} \right),\quad
			   (t,x) \in \left[ \frac{1}{2cr^n},\infty \right) \times [0,1].
			   \label{eqn:16}
			 \end{equation}
\end{itemize}
\end{theorem}

\begin{proof}
This is a consequence of \eqref{eqn:11r} and {\Fa}$_c$.
In fact, since $f\in {\mathcal P}_c$ satisfies
the inequality $f(z)\geq 0$ for $z\in [0,1]$ by \eqref{eqn:cyf},
we have \eqref{eqn:11r},
while Theorem \ref{thm:24} guarantees that {\Fa}$_c$ holds.
\end{proof}

By Theorem \ref{thm:25} we see that $H_t f$ with $f\in \mathcal{P}_c$
is a piecewise quadratic function in $[0,1]$ for all $t>0$
and that the $x$-coordinate of each vertex of the parabolas
making up $H_tf$ always belongs to $\mathbb{Q}_r$.
In general it is known that $H_t f$ for $f \in C_p(\mathbb{R})$
is $\frac{1}{2t}$-semiconcave in $[0,1]$ for all $t>0$.
For $f\in \mathcal{P}_c$ we deduce from \eqref{eqn:16} that
\[ H_t f(x)-\frac{x^2}{2t}
=\frac{1}{2t} \min_{k \in \{ 0,1,2,3,\dots,r^n \}}
\left[ -\frac{2k}{r^n} x+ \left( \frac{k}{r^n} \right)^2 + f \left( \frac{k}{r^n} \right) \right] \]
for $(t,x) \in [ \frac{1}{2cr^n},\infty) \times [0,1]$.
This shows that $H_t f(x)-\frac{x^2}{2t}$ is
not only concave but also piecewise linear in $[0,1]$.


One may ask if, conversely,
a function $f \in C_p(\mathbb{R})$ satisfying {\Fb}$_c$
for some $c>0$ is nowhere differentiable.
We have no complete answer to this question at the moment.
However, we can prove that such an $f$ is non-differentiable
on a dense subset of $\mathbb{R}$.
In general this is not enough to infer that it is nowhere differentiable,
as is shown by the Riemann function.
Indeed, let $R$ be the Riemann function defined by
\[ R(x)= \sum_{j=1}^{\infty} \frac{\sin (\pi j^2 x)}{j^2},
\quad x \in \mathbb{R}. \]
Set
\[ F:=\left\{ \left. \frac{2A+1}{2B+1} \, \right| \,
A, B \in \mathbb{Z} \right\} \ (\subset \mathbb{Q}). \]
By Hardy \cite{Ha} and Gerver \cite{G1,G2},
it is shown that $R$ is differentiable on the set $F$
and that $R$ is non-differetiable on the set
$(\mathbb{R} \setminus \mathbb{Q}) \cup (\mathbb{Q} \setminus F)$.


\begin{theorem}\label{thm:27}
Let $f \in C_p(\mathbb{R})$ 
and let $c>0$ be a constant.
Assume that {\Fb}$_c$ holds.
Then, there exists a dense subset of the interval $[0,1]$ such that
$f$ is non-differentiable at each point of this subset.
\end{theorem}

We denote by $D^- f(x)$ the \textit{subdifferential} of $f$ at $x$, that is,
the set of $\phi'(x)$ such that $\phi \in C^1$ near $x$ and
$f-\phi$ has a local minimum at $x$.
We list basic properties of the subdifferential
used in the proof of Theorem \ref{thm:27}.
Let $f \in C_p(\mathbb{R})$ and $x \in \mathbb{R}$.

\begin{enumerate}[label=(\Roman*),leftmargin=*]
\item\label{item:thm27-i}
If $f$ is differentiable at $x$, then $D^- f(x)=\{ f'(x) \}$
(\cite[Lemma II.1.8-(b)]{BCD});

\item\label{item:thm27-ii}
Let $t>0$ and choose $z \in \mathbb{R}$ such that $H_t f(x)=q_f(t,x;z)$.
Then $\frac{x-z}{t} \in D^- f(z)$ (\cite[Lemma II.4.12-(iii)]{BCD}).
\end{enumerate}

\begin{proof}[Proof of Theorem \ref{thm:27}]
Fix $x_0 \in (0,1)$ and $\varepsilon>0$,
and let $I=(x_0-\varepsilon,x_0+\varepsilon)$.
We prove that there is some $z \in I$
such that $f$ is not differentiable at $z$.
We may assume that $\varepsilon < \min\{ x_0, \, 1-x_0 \}$,
so that $I \subset [0,1]$.
Let $t \in (0, \frac{\varepsilon^2}{2M})$,
with $M>0$ the oscillation of $f$, that is,
$M=\sup_{\mathbb{R}} f-\inf_{\mathbb{R}} f$.
Since $H_t f$ is represented by \eqref{eqn:16}
with $n$ such that $t \geq \frac{1}{2cr^n}$,
there exists some $\delta \in (0,\varepsilon)$ such that
$H_t f=q_f (t,\cdot ;z)$ in $J:=[x_0-\delta,x_0] \subset I$
with $z=\frac{k}{r^n}$ for some $k \in \{ 0,1,2,3, \dots, r^n \}$.
The choice of $t$ then guarantees that $z \in I$.
Indeed, we have
\[
 f(x_0) \geq H_t f (x_0)=f(z)+\frac{1}{2t}(x_0-z)^2,
\]
and hence $(x_0-z)^2 \leq 2 t (f(x_0)-f(z))\leq 2Mt<\varepsilon^2$,
that is, $z \in I$.

It follows from \ref{item:thm27-ii} that
$\frac{x-z}{t} \in D^- f(z)$ for all $x \in J$.
This implies that
$[\frac{x_0-\delta-z}{t},\frac{x_0-z}{t}] \subset D^- f(z)$:
that is, $D^- f(z)$ is not a singleton.
Hence we conclude by \ref{item:thm27-i} that $f$ is not differentiable
 at $z$.
\end{proof}

\begin{remark}
The above proof actually shows that
the dense set we found is a subset of $\mathbb{Q}_r$.
\end{remark}

\section{Concluding remark}
\label{sec:concluding-remark}

We conclude this paper
by mentioning another possible definition of $\mathcal{P}_c$.
Let us define $\mathcal{P}'_c$
as the set of all $f \in C_p(\mathbb{R})$ such that
there exists an infinite subset $\mathbb{N}' \subset \mathbb{N}_0$
such that $f$ satisfies \eqref{eqn:23e} for all $(n,k,y) \in \nky$
with $n \in \mathbb{N}'$.
In other words, we require \eqref{eqn:23e}
only for some subsequence of $n \in \mathbb{N}_0$.
Even if this generalized class $\mathcal{P}'_c$ is used,
one can easily see that
Theorem \ref{thm:24} is obtained in a suitable sense.
Namely, $f \in \mathcal{P}'_c$ if and only if
$f$ satisfies {\Fa}$_c$ with ``For all $n \in \mathbb{N}'$''
instead of ``For all $n \in \mathbb{N}_0$''.
The proof is almost the same as before.

Moreover, Theorem \ref{thm:31} is true for a function in
$\mathcal{P}':=\bigcup_{c>0} \mathcal{P}'_c$
since the proof still works when taking the limit along $\mathbb{N}'$.
The formula \eqref{eqn:31} still gives many examples of functions in $\mathcal{P}'$.
Though $\mathcal{P}'$ provides a more general class than does $\mathcal{P}$,
there are, however, no essential changes or difficulties in the proofs.
For this reason, for simplicity of presentation,
the authors decided to give results in this paper
for $\mathcal{P}_c$ instead of $\mathcal{P}'_c$.

\section*{Acknowledgement}
Antonio Siconolfi appreciates funding for selected research
from the Faculty of Science, University of Toyama.
It enabled him to visit the University of Toyama in March, 2018.

\end{document}